\documentclass[a4paper, 11pt, reqno]{amsart}
\usepackage[english]{babel}
\usepackage[utf8]{inputenc}
\usepackage[T1]{fontenc}
\usepackage{amsthm, amsmath, amssymb, enumerate, enumitem, mathtools, lmodern, microtype, mathrsfs, listings, tikz-cd, comment, float}
\usepackage{a4wide}
\usepackage{hyperref}
\usepackage{nicefrac}
\newcommand{\half}{{\nicefrac{1}{2}}}

\usepackage{bbm}

\allowdisplaybreaks 

\theoremstyle{plain}
\newtheorem{thm}{Theorem}[section]
\newtheorem{prop}[thm]{Proposition}
\newtheorem{lemma}[thm]{Lemma}

\theoremstyle{definition}

\theoremstyle{remark}
\newtheorem*{rmk}{Remark}

\binoppenalty=\maxdimen
\relpenalty=\maxdimen

\numberwithin{equation}{section}
\setlist{nosep}
\setlist{noitemsep}

\newcommand{\Z}{\mathbb{Z}}

\newcommand{\N}{\mathbb{N}}
\newcommand{\R}{\mathbb{R}}

\DeclareMathOperator{\lcm}{lcm}

\newcommand{\Fc}{\mathcal{F}}
\newcommand{\Gc}{\mathcal{G}}

\newcommand{\Sc}{\mathcal{S}}

\newcommand{\low}{\mathrm{low}}
\newcommand{\up}{\mathrm{up}}

\makeatletter
\def\paragraph{\@startsection{paragraph}{4}%
  \z@\z@{-\fontdimen2\font}%
  {\normalfont\bfseries}}
\makeatother

\newcommand{\modi}[1]{\, \mathrm{mod} \, #1}

\usepackage{enumitem}
\setlist[enumerate]{label=\rm\roman*),itemindent=0pt,leftmargin=1cm,itemsep=0.5ex}

\title[Smallest Counterexample to Log-Concavity of D'Arcais Polynomials]{On the Smallest Counterexample to the \\ Log-Concavity of the D'Arcais Polynomials}

\author[Charlton]{Steven Charlton}
\address{}
\email{mail@stevencharlton.net}

\author[Heim]{Bernhard Heim}
\author[Stumpenhusen]{Johann Stumpenhusen}
\address{Department of Mathematics and Computer Science, Division of Mathematics, University of Cologne,
	Weyertal 86-90, 50931 Cologne, Germany}
\email{bheim@uni-koeln.de}
\email{jstumpen@math.uni-koeln.de}

\date{June 18, 2026}

\begin{document}

\begin{abstract}
    Recently, Starr used asymptotic methods to disprove a conjecture by Heim--Neuhauser and Abdesselam about the log-concavity of the D'Arcais polynomials, without giving an explicit counterexample.  We refine the asymptotics, to give the necessary estimates on convolutions of \( \sigma_{-1} \), and identify the first counterexample at \( \lambda = 65\,214\,507\,758\,400 \).  We also consider the asymptotic density of such counterexamples.
\end{abstract}

\maketitle

\section{Introduction and statement of results}

The \emph{D'Arcais polynomials}  \cite{DAr13} (or Nekrasov-Okounkov polynomials \cite{Ha10,NO06,Zh22} in combinatorics) are a sequence of polynomials given by
\begin{equation}\label{eqn:darcais:exp}
	\sum_{n=0}^\infty P_n^\sigma(X) q^n \coloneqq \prod_{m=1}^\infty (1 - q^m)^{-X} = \exp\Bigg( X \sum_{j=1}^\infty \sigma(j) \frac{q^j}{j} \Bigg) \,,
\end{equation}
where \( \sigma(n) = \sigma_1(n) \), with \( \sigma_a(n) \coloneqq \sum_{d \mid n} d^a \) the generalised sum-of-divisors function.  These polynomials are important in understanding the Fourier coefficients of \( \eta(q)^{-X} \); the Lehmer conjecture \cite{Leh47} on the non-vanishing of the Ramanujan \( \tau \)-function is equivalent to the conjecture that \( P_n^\sigma(X) \) has no root at \( X = -24 \).  The polynomials have been generalised to replace \( \sigma \) with an arbitrary arithmetic function \( g \colon \mathbb{N} \longrightarrow \mathbb{Z} \), with \( g(1) = 1 \), connecting to areas such the Weil conjectures \cite[Section 2.1]{MM13} and the study of subgroups of a given index in $\Z^\ell$ \cite{Abdessel23}.

We are interested in the log-concavity of the polynomial \( P_n^\sigma(X) \).  Recall that a sequence \( (a_n)_{n=1}^\infty \) of real numbers is said to be \emph{log-concave at $k$} if
\[a_k^2 \geq a_{k + 1}a_{k-1} \,. \]  
For any polynomial \( H(X) = \sum_{i=0}^d h_i X^i \), the finite sequence \( (h_i)_{i=0}^d\) of coefficients can be extended by 0 to obtain an infinite sequence.  We say \( H(X) \) is log-concave at \( k \) if this extended sequence is log-concave at $k$.  In their study of the D'Arcais polynomials, Heim and Neuhauser \cite[Challenge 3]{HeimNeu20} conjectured that they are always log-concave. Abdesselam further generalised this  \cite[Conjecture 1.1]{Abdessel23} to a certain family of polynomials $P_n^{g_\ell}(X) $ related to the number of subgroups of a given index in $\Z^\ell$ (here $g_2 = \sigma$, and $g_\ell$ is some recursively defined generalisation).  Starr examined this case, and disproved the conjecture for $\ell = k = 2$.

Extracting the coefficient of \( X^k \) in Eqn.~\eqref{eqn:darcais:exp}, via the Taylor series of \( \exp(X) \), shows
\[
	\sum_{n=0}^\infty p_n^\sigma(k) q^n = \frac{1}{k!} \Bigg( \sum_{j=1}^\infty \frac{\sigma(j)}{j} q^j \Bigg)^k = \frac{1}{k!}  \Bigg(  \sum_{j=1}^\infty \sigma_{-1}(j) q^j \Bigg)^k \,,
\]
Hence
\[
	p_n^\sigma(k) = \frac{1}{k!} \sum_{\substack{j_1 + \cdots + j_k = n \\ j_k \geq 1}} \sigma_{-1}(j_1) \sigma_{-1}(j_2) \cdots \sigma_{-1}(j_k) = \frac{1}{k!} (\underbrace{\sigma_{-1} \ast \cdots \ast \sigma_{-1}}_{k})(n) \,,
\]
where \( (a \ast b)(n) \coloneqq \sum_{i=1}^{n-1} a(i) b(n-i) \) is the Cauchy (or additive) convolution of two arithmetic functions \( a,b \colon \mathbb{N} \longrightarrow \mathbb{C} \). Starr used asymptotic estimates on \( \sigma_{-1} \) convolutions investigate the log-concavity of \( P_n^\sigma(X) \) at \( k = 2, 3, 4, \ldots \).  Starr gives the following asymptotic.
\begin{lemma}[Starr {\cite[Theorem 2.1]{Starr}}]\label{lem:Starr}
	Let $a, b \geq 0$ and $r, s \geq 1$, then we have
	\begin{multline*}
	\sum_{k = 1}^{n - 1} k^a \sigma_{-r}(k) (n - k)^b \sigma_{-s}(n - k) \sim \\[-1ex]
	\frac{\Gamma(a + 1)\Gamma(b + 1)}{\Gamma(a + b + 2)} \cdot \frac{\zeta(r + 1)\zeta(s + 1)}{\zeta(r + s + 2)} \cdot n^{a + b + 1}\sum_{d \mid n} d^{-r-s-1}, \qquad \text{as } n \to \infty.
	\end{multline*}
\end{lemma}
This is connected to sums \( \sum_{k=1}^{n-1} \sigma_r(k) \sigma_s(n-k) \) considered by Ramanujan \cite{Ram16}, involving sum-of-divisor functions with \( r,s \) positive odd integers.  In Starr's case, more precise asymptotics, with the leading order remainder terms, are given by \cite{Hal49,Ing27,Oliver23}.  From this asymptotic, Starr calculates the limiting behaviour of the ratio \( \tfrac{p_n^\sigma(k)^2}{p_n^\sigma(k-1) p_n^\sigma(k+1)} \), which controls the log-concavity of \( P_n^\sigma(X) \) at \( k \).
\begin{thm}[Starr {\cite[Corollary 2.4]{Starr}}]\label{thm:Starr}
The following limit behaviour holds
	\begin{align*}
	\liminf_{n \to \infty} \frac{p_n^{\sigma}(2)^2}{p_n^{\sigma}(3)p_n^{\sigma}(1)} &= 0 \,, \\
    \intertext{and for $3 \leq k \leq n - 1$,}
	\liminf_{n \to \infty} \frac{p_n^{\sigma}(k)^2}{p_n^{\sigma}(k+1)p_n^{\sigma}(k-1)} &= \frac{k}{k-1} \cdot \frac{\zeta(2k+2)\zeta(2k-2)}{\zeta(2k)^2} \cdot \frac{\zeta(2k-1)^2}{\zeta(2k-3)\zeta(2k+1)} > 1 \,.
	\end{align*}
\end{thm}
In particular \( P_n^\sigma(X) \) is eventually not log-concave at \( k=2 \) for some infinite list of \( n \) but is log-concave at \( k \geq 3 \), for all but finitely many \( n \).  We establish that the first counterexample to the log-concavity at $k=2$ occurs at \( \lambda = 65\,214\,507\,758\,400 = 29\sharp \cdot  2^5 \, 3^2 \, 5^1 \,  7^1 \) (the 80\textsuperscript{th} superabundant number),  where \( m\sharp \) is the primorial of $m$ (the product of primes \( \leq m \)).
\begin{thm}\label{thm:result}
    Let $\lambda = 65\,214\,507\,758\,400$.  Then \( n = \lambda \) is  first counterexample to the log-concavity of $P_n^\sigma(X)$ at $k=2$, meaning
    \begin{enumerate}
    \item the polynomial $P_\lambda^\sigma(X)$ is not log-concave at $k=2$, i.e.,
    \begin{equation}\label{eqn:p22p1p3}
        \frac{(p_\lambda^{\sigma}(2))^2}{p_\lambda^{\sigma}(1) \cdot p_\lambda^{\sigma}(3)} < 1 \,, \text{ and}
    \end{equation} 
    \item for any $m<\lambda$, the polynomial $P_m^\sigma(X)$ is log-concave at $k=2$.
    \end{enumerate}
\end{thm}

We prove the two parts of Theorem \ref{thm:result} separately.  For part i), the first step is to refine the asymptotics Starr gives.  In Section \ref{sec:bounds} we independently establish explicit bounds on the 2-fold and 3-fold convolutions of \( \sigma_{-1} \) and thereby on \( p_n^\sigma(2) \) and \( p_n^\sigma(3) \).  These bounds involve only simple combinations of sum-of-divisor functions (no convolutions), so can be computed efficiently.  The claim in Eqn.~\eqref{eqn:p22p1p3} is then a direct verification.  Part ii) requires a structured approach and significant computer time to exhaustively check no smaller values give a counterexample.  We outline the strategy in Section \ref{sec:minimality}.  

Finally, the methods used to bound the fraction in Eqn.~\eqref{eqn:p22p1p3} culminate in bounding the density of counterexamples.

\begin{thm}\label{thm:density}
    The following bounds on the asymptotic density of counterexamples hold:
    \begin{align*}
        \delta_\mathrm{sup} & \coloneqq \limsup_{x \to \infty} \frac{\#\left\{n \leq x : \frac{(p_n^{\sigma}(2))^2}{p_n^{\sigma}(1) \cdot p_n^{\sigma}(3)} < 1 \right\}}{x} \leq 0.000679406 
    \intertext{and}
    \delta_\mathrm{inf} & \coloneqq \liminf_{x \to \infty} \frac{\#\left\{n \leq x : \frac{(p_n^{\sigma}(2))^2}{p_n^{\sigma}(1) \cdot p_n^{\sigma}(3)} < 1 \right\}}{x} \geq 2.47323 \ldots \times 10^{-16}  \,.
    \end{align*}
\end{thm}

\begin{rmk}
    This implies that, if the asymptotical density $\delta$ of counterexamples to the log-concavity of $P_n^\sigma(X)$ at $k = 2$ exists, then it will satisfy
    \[ 2.47323 \ldots \times 10^{-16}  \leq \delta \leq 0.000679406 \, .\]
    In particular, it will be both positive and strictly less than 1. In any case, $P_n^\sigma(X)$ is log-concave at $k=2$ for infinitely many \( n \) and not log-concave at $k=2$ for infinitely many other \( n \).
\end{rmk}

We end the introduction by recalling some notation and basic results.

\subsection*{Notation and basic results}

For  $m, n \in \N$ we denote their greatest common divisor by $(m,n)$ and their least common multiple by $[m,n]$.
Some of the constants we use are derived from quotients Riemann $\zeta$-values: fix parameters \( r, s > 1 \) and \(\mu \in \N\), and write
	\[
	C_\mu^{(r,s)} \coloneqq \sum_{\substack{d',e'=1 \\ (d',e') = 1}}^{\mu-1} \frac{1}{d'^r e'^s} \,,
	\]
	for the partial sum to \( (\mu-1) \)-terms. From \cite[Lemma 3.3]{Oliver23}, we recall that
	\begin{equation}\label{eqn:gcd1sum}
	\lim_{\mu \to \infty} C_\mu^{(r,s)} = \sum_{\substack{d',e'=1 \\ (d',e') = 1}}^{\infty} \frac{1}{d'^r e'^s} = \frac{\zeta(r)\zeta(s)}{\zeta(r+s)} \,.
	\end{equation}
	We also recall  the following basic estimate from integral comparison
	\begin{equation}\label{eqn:harmonic}
	\log(n) \leq \sum_{i=1}^{n-1} i^{-1} \leq \log(n) + 1 \,.
	\end{equation}	
    
\section{Upper and lower bounds on \texorpdfstring{\( p_2^\sigma(n) \)}{p\_2\textasciicircum{}{}sigma(n)} and \texorpdfstring{\( p_3^\sigma(n) \)}{p\_3\textasciicircum{}{}sigma(n)}}\label{sec:bounds}

The main objective in this section lies in proving part i) of Theorem \ref{thm:result} which we accomplish by using the following proposition, with suitable parameters.

\begin{prop}\label{prop:Bounds}
    Let $\mu \in \N$ be arbitrary and define
    \begin{align*}
        p_n^\sigma(2)_{\up} &\coloneqq \frac{1}{2!} \left(\frac{5}{2} (n-1) \sigma_{-3}(n) + (1 + \log(n))^2\right) \,;\\[1ex]
        p_n^\sigma(2)_{\low,\mu} &\coloneqq \frac{1}{2!} \left(C_\mu^{(2,2)} \cdot (n-1) \sigma_{-3}(n) - C_\mu^{(2,2)} \cdot (n-1) \cdot 2n^{\half} \Big(\frac{\mu}{n}\Big)^3 - (\log(n) + 1)^2\right) \,;\\[1ex]
        p_n^\sigma(3)_{\up} &\coloneqq \begin{aligned}[t]\frac{1}{3!}\bigg( \frac{35(n-1)^2}{16} \sigma_{-5}(n) & {} + \frac{15(n-1)}{4} \sigma_{-4}(n) \zeta(3) (1 + \log(n)) \\[-1ex]
	    &{} + \frac{5}{2} (n-1) \sigma_{-3}(n) \zeta(2)  + n (1 + \log(n))^3\bigg)\, ;\end{aligned}\\[1ex]
        p_n^\sigma(3)_{\low,\mu} &\coloneqq \frac{1}{3!}\bigg(\begin{aligned}[t] \frac{(n-1)^2}{2} \sigma_{-5}(n) & C_\mu^{(2,2)} C_\mu^{(4,2)} - (n-1)^2 n^{\half} \Big(\frac{\mu}{n}\Big)^5 C_\mu^{(2,2)}  C_\mu^{(4,2)}\\[-0.5ex]
    & {} - \frac{3(n-1)}{2} C_\mu^{(2,2)}  \sigma_{-4}(n) \zeta(3) (1 + \log(n))\\[-0.5ex]
    &{} - C_\mu^{(2,2)} \mu^2 \zeta(2) (1 + \log(n)) - n(1+\log(n))^3\bigg) \,.\end{aligned}
    \end{align*}
    Then, for all $n \in \N$, we have
    \begin{align*}
        p_n^\sigma(2)_{\up} \geq p_n^\sigma(2) &\geq p_n^\sigma(2)_{\low,\mu}\, ,\\*
        p_n^\sigma(3)_{\up} \geq p_n^\sigma(3) &\geq p_n^\sigma(3)_{\low,\mu} \,.
    \end{align*}
\end{prop}

\begin{rmk}
    Obtaining effective bounds requires careful manipulation of the given Cauchy convolutions. A simpler lower bound for $p_n^\sigma(2)$ may be obtained via the ansatz
    \[p_n^\sigma(2) = \frac{1}{2!} \sum_{k = 1}^{n-1} \sigma_{-1}(k)\sigma_{-1}(n - k) \geq \frac{2}{n^2} \sum_{k = 1}^{n - 1}\sigma(k)\sigma(n - k) = \frac{1}{6n^2} \left(5\sigma_3(n) - (6n - 1)\sigma(n)\right)\]
    where we have used that $k(n - k)$ attains its maximum at $k = \frac{n}{2}$ and have invoked a well-known identity of Eisenstein series. However, its leading term is $\sim {\frac{5}{6}}n\sigma_{-3}(n)$ while the leading term of $p_n^\sigma(2)_{\low,\mu}$ is $\sim \smash{C_\mu^{(2,2)}}n\sigma_{-3}(n)$, where the constant $\smash{C_\mu^{(2,2)}}$ can be taken arbitrarily close to $\frac{5}{2}$, the constant given by Lemma \ref{lem:Starr}.
\end{rmk}

The proof of Proposition \ref{prop:Bounds} is split into smaller pieces within the subsections to follow.

\subsection{Bounding sums over residue classes}

	Fix \( f: \N \longrightarrow \R \) a general function. We are interested in sums of the type
	\[
	\sum_{\substack{k = 1 \\ d \mid k \,, e \mid n - k}}^{n-1} f(k) \, .
	\]
	An argument in \cite{Starr} shows that if \( (d,e) \nmid n \), then no  \( k \) satisfy the conditions \( d \mid k, e \mid n -k \).  Otherwise \( (d,e) \mid n \), and all such \( k \) satisfying \( d \mid k , e \mid n-k \) are given by \( k \equiv k_0 \modi{[d,e]} \), where \( k_0 = xnd / (d,e) \) for \( (d,e) = x d + y e \) a solution given by the extended Euclidean algorithm.  Letting $\mathbbm{1}_\mathbb{N}$ be the indicator function of $\N$ on the rationals, this means
	\begin{equation}
	\sum_{\substack{k = 1 \\ d \mid k \,, e \mid n - k}}^{n-1} f(k) \, = \sum_{\substack{k=1 \\ k \equiv k_0 \modi{[d,e]}}}^{n-1} f(k)\mathbbm{1}_\mathbb{N}\left(\frac{n}{(d,e)}\right)\,.
	\end{equation}
	
    Even though the examination of these sums would be an interesting field of study in itself, we will focus on the two specific cases appearing in our computations.
    
	\begin{lemma}\label{lem:SumOfInnerFunction1}
	    Define
        \[\Fc(n,d,e) \coloneqq \sum_{\substack{k = 1 \\ d \mid k \,, e \mid n - k}}^{n-1} 1 \, .\]
        We have
	\[
	\left(\frac{n-1}{[d,e]} - 1\right) \cdot \mathbbm{1}_\mathbb{N}\left( \frac{n}{(d, e)} \right) \leq \Fc(n, d, e) \leq \left( \frac{n-1}{[d,e]} + 1 \right) \mathbbm{1}_\mathbb{N}\left( \frac{n}{(d, e)} \right) \,.
	\]
	\end{lemma} 
	\begin{proof}
	    This is because there are \( \lfloor \frac{n-1}{[d,e]} \rfloor \) complete intervals of length \( [d,e] \), each of which contributes a solution, plus a potential solution from the final incomplete interval.  Since \( \frac{n-1}{[d,e]} -1 \leq \lfloor \frac{n-1}{[d,e]} \rfloor \) and \( \lfloor \frac{n-1}{[d,e]} \rfloor + 1 \leq \frac{n-1}{[d,e]} + 1 \), we obtain the claimed bounds.
	\end{proof} 
    In what follows, it is convenient to use the following slightly weaker bounds
	\begin{equation}\label{eqn:p2inner:bounds}
	\frac{n-1}{[d,e]} \cdot \mathbbm{1}_\mathbb{N}\biggl( \frac{n}{(d, e)} \biggr) - 1 \leq \Fc(n, d, e) \leq \frac{n-1}{[d,e]} \cdot \mathbbm{1}_\mathbb{N}\biggl( \frac{n}{(d, e)} \biggr) + 1 \,.
	\end{equation}
	
	\begin{lemma}\label{lem:SumOfInnerFunctionK-1}
	    Define
        \[\Gc(n,d,e) \coloneqq  \sum_{\substack{k = 1 \\ d \mid k \,, e \mid n - k}}^{n-1} k - 1 \, .\]
        We have
        \begin{align*}
	\biggl( -\frac{3(n-1)}{2} + \frac{(n-1)^2}{2[d,e]} \biggr) \cdot \mathbbm{1}_\mathbb{N}\biggl( \frac{n}{(d, e)} \biggr) & \leq \Gc(n,d,e) \\
    &\leq \biggr( [d,e] + \frac{3(n-1)}{2} + \frac{(n-1)^2}{2[d,e]} \biggr) \cdot \mathbbm{1}_\mathbb{N}\biggl( \frac{n}{(d, e)} \biggr) \, .
	\end{align*}
	\end{lemma}
    
    \begin{proof}
        Again
	\begin{align*}
	\Gc(n,d,e) = \mathbbm{1}_\mathbb{N}\biggl( \frac{n}{(d, e)} \biggr) \cdot \sum_{\substack{k=1 \\ k \equiv k_0 \modi{[d,e]}}}^{n-1} (k-1) 
	\end{align*}
	Let \( M \coloneqq [d,e] \), and \( \alpha = \lfloor (n-1) / M \rfloor \). If $M \geq n$, it is obvious that
    \begin{align*}
        \sum_{\substack{k=1 \\ k \equiv k_0 \modi{[d,e]}}}^{n-1} (k-1) \leq k_0 - 1 \leq [d,e] \leq [d,e] + \frac{3(n-1)}{2} + \frac{(n-1)^2}{2[d,e]} \, .
    \end{align*}
    Otherwise, by taking the maximum value of \( k_0 = M \) in the interval \( [1, \ldots, M] \), we find
	\begin{align*}
	\sum_{\substack{k=1 \\ k \equiv k_0 \modi{[d,e]}}}^{n-1} (k-1) & \leq \sum_{\substack{k=1 \\ k \equiv [d,e] \modi{[d,e]}}}^{n} (k-1) \\
	& \leq \sum_{i=1}^{\alpha+1} (i M - 1) \\
	& = \frac{(\alpha+1)(\alpha+2)}{2} M \\
	& \leq M + \frac{3(n-1)}{2} + \frac{(n-1)^2}{2M} \\
	& = [d,e] + \frac{3(n-1)}{2} + \frac{(n-1)^2}{2[d,e]}
	\end{align*}
	Conversely, by taking the minimum value of \( k_0 = 1 \) in the interval \( [1,\ldots,M] \),  we can find a lower bound.
    Since \( \alpha > (n-1)/M - 1 \), a lower bound is certainly
	\begin{align*}
     \geq \sum_{i=1}^{\alpha+1} ((i-1)M) - (n-1) 
	& = \frac{\alpha(\alpha+1)}{2} M - (n-1) \\
	& \geq \frac{(n-1)^2}{2M} - \frac{3(n-1)}{2} \\
	& = \frac{(n-1)^2}{2[d,e]} - \frac{3(n-1)}{2} \,.
	\end{align*}
	This yields the claim.
    \end{proof}
	
	\subsection{Upper bound on \texorpdfstring{\( p_n^\sigma(2) \)}{p\_n\textasciicircum{}sigma(2)}}  For \( k = 2 \), we write out the convolution and apply explicit estimates.  We have
	\[
	2! \cdot p_n^\sigma(2) \, = \, \sum_{k=1}^{n-1} \sigma_{-1}(k) \sigma_{-1}(n-k)
	\, = \, \sum_{k=1}^{n-1} \sum_{\substack{d \smash{\mid} k \\ e \smash{\mid} n-k}} \frac{1}{d e} \, = \, \sum_{d,e=1}^{n-1} \frac{1}{d e} \Fc(n,d,e)
	\]
	as in Lemma \ref{lem:SumOfInnerFunction1}. Applying Eqn.~\eqref{eqn:p2inner:bounds} gives
	\begin{align*}
	2! \cdot p_n^\sigma(2) & \leq \sum_{d,e=1}^{n-1} \frac{1}{d e} \Bigg( \frac{n-1}{[d,e]} \cdot \mathbbm{1}_\mathbb{N}\biggl( \frac{n}{(d, e)} \biggr) + 1 \Bigg)\\
	& \leq \sum_{d,e=1}^{n-1} \frac{1}{d e} \frac{n-1}{[d,e]} \cdot \mathbbm{1}_\mathbb{N}\biggl( \frac{n}{(d, e)} \biggr) + (1 + \log(n))^2 \,,
	\end{align*}
	by Eqn.~\eqref{eqn:harmonic}.  Set \( w = (d,e) \mid n \), with \( d = d' w \) and \( e = e' w \), so that \( (d',e') = 1 \) and \( [d,e] = w d' e' \).  Therefore
	\begin{align*}
	2! \cdot p_n^\sigma(2) & \leq (n-1) \sum_{w\mid n} \sum_{\substack{d',e'=1 \\ (d',e') = 1}}^{n/w-1} \frac{1}{w^3 d'^2 e'^2} + (1 + \log(n))^2 \\
	& \leq (n-1) \sum_{w\mid n} \frac{1}{w^3} \sum_{\substack{d',e'=1 \\ (d',e') = 1}}^{\infty} \frac{1}{d'^2 e'^2} + (1 + \log(n))^2  \\
	& = \lim_{\mu \to \infty} C_{\mu}^{(2,2)}  \cdot (n-1) \sum_{w\mid n} \frac{1}{w^3}  + (1 + \log(n))^2  \\
	& = \frac{5}{2} (n-1) \sigma_{-3}(n) + (1 + \log(n))^2 \,,
	\end{align*}
	where \( \lim_{\mu\to\infty} C_\mu^{(2,2)} = \frac{5}{2} \) from Eqn.~\eqref{eqn:gcd1sum}.
	
	\subsection{Lower bound on \texorpdfstring{\( p_n^\sigma(2) \)}{p\_n\textasciicircum{}sigma(2)}}
	For the lower bound on \( p_n^\sigma(2) \), note that \( C_\mu^{(2,2)} < \frac{5}{2} \), and is monotonically increasing as \( \mu \to \infty \).  By Eqns.~\eqref{eqn:p2inner:bounds} and \eqref{eqn:harmonic}, we again have
	\begin{align*}
	2! \cdot p_n^\sigma(2) &= \sum_{d,e=1}^{n-1} \frac{1}{d e} \Fc(n,d,e)  \\
	&\geq \sum_{d,e=1}^{n-1} \frac{1}{d e} \Bigg( \frac{n-1}{[d,e]} \cdot \mathbbm{1}_\mathbb{N}\biggl( \frac{n}{(d, e)} \biggr) - 1 \Bigg)\\
	&\geq \sum_{d,e=1}^{n-1} \frac{1}{d e} \frac{n-1}{[d,e]} \cdot \mathbbm{1}_\mathbb{N}\biggl( \frac{n}{(d, e)} \biggr) - ( \log(n) + 1)^2 
	\end{align*}
	Set \( w = (d,e) \mid n \), with \( d = d' w \) and \( e = e' w \), so that \( (d',e') = 1 \) and \( [d,e] = w d' e' \).  Then
	\begin{align}
	2! \cdot p_n^\sigma(2) &\geq (n-1) \sum_{w\mid n} \frac{1}{w^3} \sum_{\substack{d',e'=1 \\ (d',e') = 1}}^{n/w-1} \frac{1}{d'^2 e'^2} - (\log(n) + 1)^2 \notag \\
	&\geq (n-1) \sum_{w\mid n} \frac{C_\mu^{(2,2)}}{w^3} - (n-1) \sum_{\substack{w\mid n \\ n/w < \mu}} \frac{C_\mu^{(2,2)}}{w^3} - (\log(n) + 1)^2  \label{eqn:p2penult} \\
	&\geq C_\mu^{(2,2)} \cdot (n-1) \sigma_{-3}(n) - C_\mu^{(2,2)} \cdot (n-1) \cdot 2n^{\half} \Big(\frac{\mu}{n}\Big)^3 - (\log(n) + 1)^2 \,, \notag
	\end{align}
	by splitting the sum over \( w \) according to whether \( n/w \geq \mu \) or \( n/w < \mu \). 
    More precisely, we have 
    \begin{align*}
    \sum_{w \mid n} \frac{C_{n/w}^{(2,2)}}{w^3} 
     \geq \sum_{\substack{w \mid n \\ n/w \geq \mu}} \frac{C_{\mu}^{(2,2)}}{w^3}  
 = \sum_{\substack{w \mid n}} \frac{C_{\mu}^{(2,2)}}{w^3} - \sum_{\substack{w \mid n \\ n/w < \mu}} \frac{C_{\mu}^{(2,2)}}{w^3} \,.
    \end{align*}
    Note above, we also use the elementary estimate \( \sum_{w\mid n} 1 = \sigma_0(n) \leq 2 n^\half \).
	
	\subsection{Upper bound on \texorpdfstring{\( p_n^\sigma(3) \)}{p\_n\textasciicircum{}sigma(3)}}
	
	For \( k = 3 \), we write out the convolution and apply  estimates on \( p_{n-k}^\sigma(2) \).  We have
	\begin{align*}
	3! \cdot p_n^\sigma(3) & = \sum_{k=1}^{n-1} 2! \cdot p_{k}^\sigma(2) \sigma_{-1}(n-k) \\
	& \leq \sum_{k=1}^{n-1} \Big( \frac{5}{2} (k-1) \sigma_{-3}(k) + (1 + \log(k))^2\Big) \sigma_{-1}(n-k)  \\
	& = \frac{5}{2} \sum_{k=1}^{n-1}  (k-1) \sigma_{-3}(k) \sigma_{-1}(n-k) + \sum_{k=1}^{n-1} (1 + \log(k))^2 \sigma_{-1}(n-k)  \,.
	\end{align*}
	
	We treat each term separately.  For the latter
	\begin{align*}
	\sum_{k=1}^{n-1} (1 + \log(k))^2 \sigma_{-1}(n-k) &= \sum_{d=1}^{n-1} \sum_{\substack{k=1 \\ d \mid n-k}}^{n-1} (1 + \log(k))^2 \frac{1}{d}  \\
	&\leq \sum_{k=1}^{n-1} (1 + \log(k))^2 \sum_{d=1}^{n-1} \frac{1}{d} \\
	&\leq \sum_{k=1}^{n-1} (1 + \log(n))^2 \sum_{d=1}^{n-1} \frac{1}{d} \\
	& \leq n (1 + \log(n))^3 \,.
	\end{align*}
	For the former, invoke Lemma \ref{lem:SumOfInnerFunctionK-1}. It holds that
	\begin{align*}
	\sum_{k=1}^{n-1}  (k-1) \sigma_{-3}(k) \sigma_{-1}(n-k)
	& = \sum_{d,e=1}^{n-1} \sum_{\substack{k=1 \\ d \mid k \\ e \mid n-k}}^{n-1} \frac{k-1}{d^3 e} \\
	& = \sum_{d,e=1}^{n-1} \frac{1}{d^3 e} \Gc(n,d,e) \\
	& \leq \sum_{d,e=1}^{n-1} \frac{1}{d^3 e} \Big( \frac{(n-1)^2}{2[d,e]} + \frac{3(n-1)}{2} + [d,e] \Big) \mathbbm{1}_\mathbb{N}\biggl( \frac{n}{(d, e)} \biggr)  \,.
	\end{align*}
	Again set \( w = (d,e) \mid n \), with \( d = d'w \) and \( e = e' w \), where \( (d',e') = 1 \) and \( [d,e] = d' e' w \).  Then we have
	\begin{align*}
	& \sum_{k=1}^{n-1}  (k-1) \sigma_{-3}(k) \sigma_{-1}(n-k) \\
	&\leq  \sum_{w \mid n} \sum_{d',e'=1}^{n/w-1} \frac{1}{d'^3 e' w^4} \bigg( \frac{(n-1)^2}{2d' e' w} + \frac{3(n-1)}{2} +  d' e' w\bigg) \\
	& \leq \sum_{w \mid n} \sum_{\substack{d',e'=1 \\ (d',e') = 1}}^{n/w-1} \bigg( \frac{(n-1)^2}{2 d'^4 e'^2 w^5} + \frac{3(n-1)}{2 d'^3 e' w^4}+  \frac{1}{d'^2 w^3} \bigg) \\
	& \leq \frac{(n-1)^2}{2} \lim_{\mu \to \infty} C_\mu^{(4,2)}  \cdot \sum_{w \mid n} \frac{1}{w^5}
    + \frac{3(n-1)}{2} \sum_{w \mid n} \frac{1}{w^4} \sum_{d'=1}^{\infty} \frac{1}{d'^3} \sum_{e'=1}^{n-1} \frac{1}{e'} + (n-1) \sum_{w \mid n} \frac{1}{w^3} \sum_{d'=1}^\infty \frac{1}{d'^2}  \\
	&= \frac{(n-1)^2}{2} \cdot \frac{7}{4} \sigma_{-5}(n) + \frac{3(n-1)}{2} \sigma_{-4}(n) \zeta(3) (1 + \log(n)) + (n-1) \sigma_{-3}(n) \zeta(2) \,,
	\end{align*}
	where \( \lim_{\mu \to \infty} C_\mu^{(4,2)}= \frac{7}{4} \) from Eqn.~\eqref{eqn:gcd1sum}.
	
	Overall, we obtain
	\begin{align*}
	3! \cdot p_n^\sigma(3) \leq {} \frac{35(n-1)^2}{16} \sigma_{-5}(n) + \frac{15(n-1)}{4} \sigma_{-4}(n) \zeta(3) (1 + \log(n)) & \\
	\qquad {} + \frac{5}{2} (n-1) \sigma_{-3}(n) \zeta(2)  + n (1 + \log(n))^3 &
	\end{align*}
	
	\subsection{Lower bound for \texorpdfstring{\( p_n^\sigma(3) \)}{p\_n\textasciicircum{}sigma(3)}}
	For \( k = 3 \), we write out the convolution and apply  estimates on \( p_{n-k}^\sigma(2) \).  Fix a parameter \( \mu > 0 \), we take the second to last equality for \( p_k^\sigma(2) \) in Eqn.~\eqref{eqn:p2penult}, to obtain a better estimate here, giving
	\begin{align*}
	3! \cdot p_n^\sigma(3) & = \sum_{k=1}^{n-1} 2! \cdot p_{k}^\sigma(2) \sigma_{-1}(n-k) \\
	& \geq \sum_{k=1}^{n-1} \Big( C_\mu^{(2,2)} \cdot (k-1) \sigma_{-3}(k) - (k-1) \sum_{\substack{w\mid k \\ k/w < \mu}} \frac{C_\mu^{(2,2)}}{w^3}  - (\log(k) + 1)^2 \Big) \sigma_{-1}(n-k) \\
	& \geq \begin{aligned}[t]  C_\mu^{(2,2)} \sum_{k=1}^{n-1} & (k-1) \sigma_{-3}(k) \sigma_{-1}(n-k)  - \sum_{k=1}^{n-1} (k-1) \sum_{\substack{w\mid k \\ k/w < \mu}} \frac{C_\mu^{(2,2)}}{w^3} \sigma_{-1}(n-k)\\[-2ex]
    &- \sum_{k=1}^{n-1} (\log(k) + 1)^2 \sigma_{-1}(n-k) \end{aligned}
	\end{align*}
	Handling each sum separately, we find
	\begin{align*}
	& \sum_{k=1}^{n-1} (\log(k) + 1)^2 \sigma_{-1}(n-k)
	\leq n (1 + \log(n))^3 \,,
	\end{align*}
	as before. Then
	\begin{align*}
	\sum_{k=1}^{n-1} (k-1) \sum_{\substack{w\mid k \\ k/w < \mu}} \frac{C_\mu^{(2,2)}}{w^3} \sigma_{-1}(n-k)
	& = \sum_{d=1}^{n-1} \sum_{\substack{k=1 \\ d \mid n-k}}^{n-1} (k-1) \sum_{\substack{w\mid k \\ k/w < \mu}} \frac{C_\mu^{(2,2)}}{w^3} \frac{1}{d} \\
	& \leq  \sum_{\substack{k=1}}^{n-1} k \sum_{\substack{w\mid k \\ k/w < \mu}} \frac{C_\mu^{(2,2)}}{w^3} \sum_{d=1}^{n-1} \frac{1}{d} \\
	& =  \sum_{\substack{k=1}}^{n-1} \sum_{\substack{w\mid k \\ k/w < \mu}} \frac{C_\mu^{(2,2)} \cdot k}{w^3} \sum_{d=1}^{n-1} \frac{1}{d} \\
	& \leq \sum_{\substack{w=1}}^{n-1} \sum_{\substack{k = 1 \\ w\mid k}}^{w \mu} \frac{C_\mu^{(2,2)} \cdot k}{w^3} \sum_{d=1}^{n-1} \frac{1}{d}\, .
    \end{align*}
    Setting $k = \alpha w$, we get
	\begin{align*}& \leq\sum_{\substack{w=1}}^{n-1} \sum_{\substack{\alpha = 1}}^{\mu} \frac{C_\mu^{(2,2)} \cdot \alpha w}{w^3} \sum_{d=1}^{n-1} \frac{1}{d} \\
	& \leq \sum_{\substack{w=1}}^{n-1} \frac{C_\mu^{(2,2)} \cdot \mu^2}{w^2} \sum_{d=1}^{n-1} \frac{1}{d} \\
	& \leq C_\mu^{(2,2)} \mu^2 \zeta(2) (1 + \log(n))
	\end{align*}
	
	Finally, from Lemma \autoref{lem:SumOfInnerFunctionK-1} as before, we have
	\begin{align*}
	& \sum_{k=1}^{n-1} (k-1) \sigma_{-3}(k) \sigma_{-1}(n-k) \\
	& = \sum_{d,e=1}^{n-1} \frac{1}{d^3 e} \mathcal{G}(n,d,e) \\
	& \geq \sum_{d,e=1}^{n-1} \frac{1}{d^3 e} \bigg( \frac{(n-1)^2}{2[d,e]} - \frac{3(n-1)}{2} \bigg) \cdot \mathbbm{1}_{\mathbb{N}}\Big( \frac{n}{[d,e]} \Big) \,.
	\end{align*}
	Set \(w = (d,e) \mid n \), with \( d = d' w \) and \( e = e' w \) with \( (d',e') = 1 \). Note \( \lim_{\mu\to\infty} C_\mu^{(4,2)} = \frac{7}{4} \) and \( C_\mu^{(4,2)} \) is monotonically increasing, from Eqn.~\eqref{eqn:gcd1sum}.
	Then
	\begin{align*}
	& \geq \sum_{w \mid n} \sum_{\substack{d',e'=1 \\ (d',e') = 1}}^{n/w-1} \frac{1}{d'^3 e' w^4} \bigg( \frac{(n-1)^2}{2e'd'w} - \frac{3(n-1)}{2} \bigg)  \\
	& = \frac{(n-1)^2}{2} \sum_{w \mid n} \sum_{\substack{d',e'=1 \\ (d',e') = 1}}^{n/w-1} \frac{1}{d'^4 e'^2 w^5} - \frac{3(n-1)}{2} \sum_{w \mid n} \sum_{\substack{d',e'=1 \\ (d',e') = 1}}^{n/w-1} \frac{1}{d'^3 e' w^4} \\
	& \geq \frac{(n-1)^2}{2} \sum_{w \mid n} \frac{1}{w^5} C_\mu^{(4,2)} - \frac{(n-1)^2}{2} \sum_{\substack{w \mid n \\ n/w < \mu}} \frac{1}{w^5} C_\mu^{(4,2)} - \frac{3(n-1)}{2} \sum_{w \mid n} \frac{1}{w^4} \sum_{\substack{d',e'=1 \\ (d',e') = 1}}^{n-1} \frac{1}{d'^3 e'} \\
	& \geq \frac{(n-1)^2}{2} \sigma_{-5}(n) C_\mu^{(4,2)} - (n-1)^2 n^\half \Big(\frac{\mu}{n}\Big)^5 C_\mu^{(4,2)} - \frac{3(n-1)}{2} \sigma_{-4}(n) \zeta(3) (1 + \log(n)) \,,
	\end{align*}
	by splitting the sum according to whether \( n/w \geq \mu \) or \( n/w < \mu \), as before.
	
	Overall, we obtain
	\begin{align*}
	3! \cdot p_n^\sigma(3) &\geq \begin{aligned}[t] \frac{(n-1)^2}{2} \sigma_{-5}(n) & C_\mu^{(2,2)} C_\mu^{(4,2)} - (n-1)^2 n^\half \Big(\frac{\mu}{n}\Big)^5 C_\mu^{(2,2)}  C_\mu^{(4,2)}\\
    &{} - \frac{3(n-1)}{2} C_\mu^{(2,2)}  \sigma_{-4}(n) \zeta(3) (1 + \log(n))\\[1ex]
    &{} - C_\mu^{(2,2)} \mu^2 \zeta(2) (1 + \log(n)) - n(1+\log(n))^3 \, ,\end{aligned}
	\end{align*}
    concluding the proof of Proposition \ref{prop:Bounds}. \hfill \qedsymbol
    
\subsection{Proof of Theorem \ref{thm:result}, part i)} Having proven Proposition \ref{prop:Bounds}, we are in a position to advance to the proof of Theorem \ref{thm:result}. For this, we fix \( \mu = 500 \). Then we obtain
	\begin{align*}
		C_{500}^{(2,2)} = 2.49452421836436753653\ldots  \,, \\
		C_{500}^{(4,2)} = 1.74791042060522673981\ldots \,.
	\end{align*}
	Note that \( \lambda = 65\,214\,507\,758\,400 = 29\sharp \cdot  2^5 \, 3^2 \, 5^1 \,  7^1 \), where \( m\sharp \) is $m$ primorial.  One can directly compute that
	\begin{align*}
		& \sigma_{-1}(\lambda) = \tfrac{68031360}{11350339} = 5.99377340183407737865\ldots \\
        & \sigma_{-3}(\lambda) = \tfrac{3367719899875296294987}{2802008927062274116000} = 1.20189477890284014791\ldots  \\
        & \sigma_{-5}(\lambda) =  \tfrac{6731466498568936993292460214710631881036243488437907}{
6491741578415157169646584347037579301727795363840000} = 1.03692767453203280218\ldots \,.
	\end{align*}
	One then has
    \begin{align*}
        p_{\lambda}^\sigma(2)_\up & = 9.79762204799613549624\ldots \times 10^{13} \\
        p_{\lambda}^\sigma(3)_{\low,500} &= 1.60236908240243543394\ldots \times 10^{27} \\
        p_{\lambda}^\sigma(1) &= 5.99377340183407737865\ldots  \,.
    \end{align*}
    Whence
	\[
		\frac{p_\lambda^{\sigma}(2)^2}{p_\lambda^{\sigma}(1) p_\lambda^{\sigma}(3)} < \frac{(p_\lambda^{\sigma}(2)_\up)^2}{p_\lambda^{\sigma}(1) \, p_\lambda^{\sigma}(3)_{\low,500}} = 0.99949007855677265035\ldots < 1 \,.
	\]
    So as claimed \( P_\lambda^\sigma(X) \) is not log-concave at \( k = 2 \). \hfill \qed

\section{Minimality of the counterexample}\label{sec:minimality}

	We establish the minimality in Theorem \ref{thm:result}, part ii) via three steps.  Firstly we check explicitly that there is no counterexample up to \( n = 10^7 \).  This requires using more efficient algorithms for computing polynomial products (or  convolutions of coefficients), to be tractable.  Then we use interval arithmetic to establish that
    \begin{equation}
    \frac{p_n^{\sigma}(2)^2}{p_n^{\sigma}(3)} \,,
	\end{equation}
     is \( >4 \) for \( n \geq 10^7 \) (indeed even for \( n \geq 2 \times 10^6 \)).
	This implies any such counterexample \( n \) must have \( \sigma_{-1}(n) > 4 \).  Finally we enumerate all such \( 1 \leq n \leq \lambda \) with the property \( \sigma_{-1}(n) > 4 \) and verify none of them give a counterexample.
	
	\subsection{Elimination of counterexamples with \texorpdfstring{\( 1 \leq n \leq 10^7 \)}{1 <= n <= 10\textasciicircum{}7}}\label{sec:smallcounter}
	
    One can implement the following procedures in any of \texttt{SageMath}, \texttt{gp/pari} or \texttt{Mathematica}, with various advantages and disadvantages in each case.  In particular, \texttt{gp/pari} is significantly faster for number theoretic routines like computing divisor sums, but lacks a native interface to list convolution routines (\texttt{ListConvolve} in \texttt{Mathematica} or \texttt{convolve} in \texttt{SageMath}) which allow the asymptotically faster \( O(n \log(n)) \) multiplication of polynomials using the Fast Fourier transform.  

   On a high-end consumer grade laptop (\texttt{\small  13th Gen Intel(R) Core(TM) i9-13980HX @ 2.2 GHz}, with 64 GB RAM), the computation of \( (p_n^\sigma(1))_{1 \leq n \leq 10^7} \) takes the following time in each computer algebra system with a single thread.
    \begin{center}
    \begin{tabular}{l|ll}
    Software & Command & Time \\ \hline
    \texttt{SageMath} & \verb|[ sigma(x,1)/x for x in range(1,10^7)]| & 45.7 seconds \\ \texttt{Mathematica} & \verb|Map[DivisorSigma[-1, #]&, Range[1,10^7]]| & 19.1 seconds \\ 
    \texttt{gp/pari} & \texttt{[ sigma(i,-1) | i <- [1..10\textasciicircum{}7]]} & \phantom{0}5.3 seconds 
    \end{tabular}
    \end{center}

    Since \( p_n^\sigma(2) = \frac{1}{2!} (\sigma_{-1} \ast \sigma_{-1})(n) \) is the Cauchy convolution of \( p_n^\sigma(1) = \sigma_{-1}(n) \), the naive method requires \( O(n^2) \) steps to compute directly, making computation beyond \( n = 10^4 \) already very time-consuming.  Using Fast Fourier Transform based algorithms one can compute polynomial multiplication (hence such Cauchy convolutions) in time \( O(n \log(n)) \).  This is available in \texttt{SageMath} via \( \texttt{convolve} \), while in \texttt{Mathematica} it is available via \texttt{ListConvolve} (with suitable options).  
    Exact rational arithmetic to compute the values of \( p_n^\sigma(2) \) and \( p_n^\sigma(3) \) is very time-consuming, and produces results with very large height, for example
    \[
     p_{50}^\sigma(3) = \frac{16481674253589243490606751802749287}{19881006138756642992377224960000}
    \,.
    \]  Therefore, for the computations of \( p_n^\sigma(2) \) and \( p_n^\sigma(3) \), we first evaluate \( p_n^\sigma(1) \) as a real number to 100 decimal places (which adds some overhead to the previous timings), and then carry out the (fast) polynomial multiplication.

    On a high-end consumer grade laptop (\texttt{\small 13th Gen Intel(R) Core(TM) i9-13980HX  @ 2.2GHz}, with 64 GB RAM), the follow are representative timings for the computation of list convolution with lists of size \( 10^7 \), in each computer algebra system.  Note compute two convolutions, the first to calculate \( p_n^\sigma(2) \) from \( p_n^\sigma(1) \), and a second to compute \( p_n^\sigma(3) \)  from \( p_n^\sigma(2) \) and \( p_n^\sigma(1) \), doubling the overall time.  Note \texttt{Mathematica} has an efficient implementation allowing computation up to \( 10^8 \) in reasonable time (with sufficient memory); the final line reports the resources required for this computation on the Universit\"at zu K\"oln departmental server (\texttt{\small Intel(R) Xeon(R) CPU E5-2687W v4 @ 3.00GHz}, with 380 GB of RAM).
     \begin{center}
    \begin{tabular}{l|l|llr}
    Software & Size & Schematic version of command & Time & Memory \\ \hline
    \texttt{gp/pari} & $10^7$ &\verb|Vec(Pol(list1) * Pol(list2))| & \phantom{0}3.9 hrs & 10.5 GB\\ 
    \texttt{SageMath} & $10^7$ & \verb|convolve(list1, list2)| & \phantom{0}1.6 hrs & 37.5 GB  \\ 
    \texttt{Mathematica} & $10^7$ & \verb|ListConvolve[list1, list2, {1, -1}, 0]| & 12.4 min &  11.6 GB \\ 
     \hline
    \texttt{Mathematica} & $10^8$ & \verb|ListConvolve[list1, list2, {1, -1}, 0]| & 12.7 hrs & 180 GB \\ %
    \end{tabular}
    \end{center}

Once the values of \( p_n^\sigma(1) \), \( p_n^\sigma(2) \) and \( p_n^\sigma(3) \) for \( 1 \leq n \leq 10^7\) are computed, it is straightforward to check that no counterexample occurs in this range.  As noted, we computed up to \( 10^8 \) in \texttt{Mathematica} at an earlier stage of the research, and can report
    \[
        \min_{1 \leq n \leq 10^8} \frac{(p_n^\sigma(2))^2}{p_n^\sigma(1) p_n^\sigma(3)} = 1.199090503\ldots \,,
    \]
    which occurs at \( n = 73\,513\,440 \).  The complete list of record minima \( 1 \leq n \leq 10^8 \) for this ratio, and the ratio value is as follows
    \begin{center}
    \begin{tabular}{rr||}
        Minimum & Ratio value \\ \hline
        3 & 10.12500\ldots \\
4 & 4.60449\ldots \\
6 & 3.23210\ldots \\
12 & 2.50636\ldots \\
24 & 2.30262\ldots \\
36 & 2.29878\ldots \\
48 & 2.23419\ldots \\
60 & 2.06896\ldots \\
120 & 1.94358\ldots \\
180 & 1.92494\ldots \\
240 & 1.88976\ldots \\
360 & 1.80650\ldots \\
720 & 1.75360\ldots \\
840 & 1.72334\ldots 
   \end{tabular}%
   \begin{tabular}{rr||}
        Minimum & Ratio value \\ \hline
1\,260 & 1.70484\ldots \\
1\,680 & 1.67076\ldots \\
2\,520 & 1.59772\ldots \\
5\,040 & 1.54807\ldots \\
10\,080 & 1.52438\ldots \\
15\,120 & 1.51057\ldots \\
25\,200 & 1.49966\ldots \\
27\,720 & 1.46924\ldots \\
55\,440 & 1.42263\ldots \\
110\,880 & 1.40023\ldots \\
166\,320 & 1.38734\ldots \\
277\,200 & 1.37709\ldots \\
332\,640 & 1.36544\ldots \\
554\,400 & 1.35533\ldots 
   \end{tabular}%
   \begin{tabular}{rr}
        Minimum & Ratio value \\ \hline
665\,280 & 1.35472\ldots \\
720\,720 & 1.32241\ldots \\
1\,441\,440 & 1.30150\ldots \\
2\,162\,160 & 1.28950\ldots \\
3\,603\,600 & 1.27994\ldots \\
4\,324\,320 & 1.26910\ldots \\
7\,207\,200 & 1.25969\ldots \\
8\,648\,640 & 1.25912\ldots \\
10\,810\,800 & 1.24806\ldots \\
21\,621\,600 & 1.22832\ldots \\
36\,756\,720 & 1.21836\ldots \\
61\,261\,200 & 1.20933\ldots \\
73\,513\,440 & 1.19909\ldots \\
\phantom{122\,522\,400} &   \\
   \end{tabular}
   \end{center}

    \begin{rmk}
   Except for the first entry (which must be \( n = 3 \), as the ratio is not defined for \( n = 1, 2 \)), the sequence of minima appears to be given by OEIS sequence \href{https://oeis.org/A004394}{A004394}, the superabundant numbers. Heuristically, one might see Cauchy convolutions of a function with itself as mollifications of increasing order. Therefore, extrema mostly depend on the lowest order convolution, which in our case equals the function \( \sigma_{-1} \) itself, whose maxima occur at the superabundant numbers.
    \end{rmk}

Below is a table of the upper bound and lower bound on the ratio \( \frac{p_n^\sigma(2)^2}{p_n^\sigma(1) p_n^\sigma(3)} \) computed with \( \mu = 2\,000 \), on the sequence \( (S_i)_{i=42}^{80} \) of superabundant numbers, heuristically giving the first counterexamples as \( \lambda = S_{80} = 65\,214\,507\,758\,400 \).
\begin{center}
\begin{tabular}{cr|cc||}
$i$ & $S_i$ & lower & upper \\ \hline
 42 & 73513440 & 1.1977\ldots & 1.2002\ldots \\
 43 & 122522400 & 1.1889\ldots & 1.1912\ldots \\
 44 & 147026880 & 1.1883\ldots & 1.1907\ldots \\
 45 & 183783600 & 1.1779\ldots & 1.1802\ldots \\
 46 & 367567200 & 1.1593\ldots & 1.1616\ldots \\
 47 & 698377680 & 1.1565\ldots & 1.1588\ldots \\
 48 & 735134400 & 1.1502\ldots & 1.1524\ldots \\
 49 & 1102701600 & 1.1497\ldots & 1.1519\ldots \\
 50 & 1163962800 & 1.1479\ldots & 1.1502\ldots \\
 51 & 1396755360 & 1.1382\ldots & 1.1404\ldots \\
 52 & 2327925600 & 1.1298\ldots & 1.1320\ldots \\
 53 & 2793510720 & 1.1293\ldots & 1.1315\ldots \\
 54 & 3491888400 & 1.1194\ldots & 1.1215\ldots \\
 55 & 6983776800 & 1.1016\ldots & 1.1038\ldots \\
 56 & 13967553600 & 1.0930\ldots & 1.0951\ldots \\
 57 & 20951330400 & 1.0925\ldots & 1.0947\ldots \\
 58 & 27935107200 & 1.0887\ldots & 1.0908\ldots \\
 59 & 41902660800 & 1.0839\ldots & 1.0861\ldots \\
 60 & 48886437600 & 1.0823\ldots & 1.0844\ldots \\
 61 & 80313433200 & 1.0729\ldots & 1.0750\ldots
  \end{tabular}%
\begin{tabular}{cr|cc}
$i$ & $S_i$ & lower & upper \\ \hline
 62 & 160626866400 & 1.0559\ldots & 1.0580\ldots \\
 63 & 321253732800 & 1.0476\ldots & 1.0496\ldots \\
 64 & 481880599200 & 1.0472\ldots & 1.0492\ldots \\
 65 & 642507465600 & 1.0435\ldots & 1.0455\ldots \\
 66 & 963761198400 & 1.0390\ldots & 1.0410\ldots \\
 67 & 1124388064800 & 1.0374\ldots & 1.0394\ldots \\
 68 & 1927522396800 & 1.0349\ldots & 1.0369\ldots \\
 69 & 2248776129600 & 1.0292\ldots & 1.0312\ldots \\
 70 & 3373164194400 & 1.0288\ldots & 1.0308\ldots \\
 71 & 4497552259200 & 1.0252\ldots & 1.0272\ldots \\
 72 & 4658179125600 & 1.0208\ldots & 1.0228\ldots \\
 73 & 6746328388800 & 1.0207\ldots & 1.0227\ldots \\
 74 & 9316358251200 & 1.0128\ldots & 1.0147\ldots \\
 75 & 13974537376800 & 1.0124\ldots & 1.0143\ldots \\
 76 & 18632716502400 & 1.0088\ldots & 1.0108\ldots \\
 77 & 27949074753600 & 1.0044\ldots & 1.0064\ldots \\
 78 & 32607253879200 & 1.0029\ldots & 1.0049\ldots \\
 79 & 55898149507200 & 1.0005\ldots & 1.0024\ldots \\
 80 & 65214507758400 & 0.9950\ldots & 0.9969\ldots \\
\phantom{81}  &  &  & 
\end{tabular}
\end{center}
	
	\subsection{Properties of potential counterexamples}\label{sec:counterprop}  We will show that for \( n \geq 10^7 \), any such counterexample must have \( \sigma_{-1}(n) > 4 \).  This will allow us to reduce the search for counterexamples to a rather sparse set of integers. \medskip

    Fix \( \mu = 500 \), with \( C_{500}^{(2,2)} = 2.49452\ldots \).  Since \( \sigma_{-3}(n) > 1 \), we note that
    \[
        p_n^\sigma(2)_{\low,500} > a(n) \coloneqq \frac{1}{2!} \left(C_{500}^{(2,2)} \cdot (n-1) - C_{500}^{(2,2)} \cdot (n-1) \cdot 2n^{\half} \bigg(\frac{500}{n}\bigg)^3 - (\log(n) + 1)^2\right) \,.
    \]
    Likewise, since \( \sigma_{-k}(n) < \zeta(k) \) for \( k = 3, 4, 5 \), we have
    \[
        p_n^\sigma(3)_{\up} < b(n) \coloneqq  \begin{aligned}[t]\frac{1}{3!}\bigg( \frac{35(n-1)^2}{16} \zeta(5) & {} + \frac{15(n-1)}{4} \zeta(4) \zeta(3) (1 + \log(n)) \\[-1ex]
	    &{} + \frac{5}{2} (n-1) \zeta(3) \zeta(2)  + n (1 + \log(n))^3\bigg)\, .\end{aligned}\\[1ex]
    \]
    We have that
    \[
        \lim_{n\to\infty} \frac{a(n)^2}{b(n)} = \frac{24 \cdot (C_{500}^{(2,2)})^2}{35 \zeta(5)} = 4.1150029176747844809\ldots \,,
    \]
    so eventually the ratio lies in an arbitrarily small interval around the limit, and in particular is \( >4 \).  For \( n \geq 2 \times 10^6 \), we compute the resulting interval explicitly.  \medskip
    
    Working with the bound \( a(n) \) on \( p_n^\sigma(2)_{\low,500} \): for \( n \geq 2 \times 10^6 \), we have that
    \begin{align*}
        \big\lvert \tfrac{1}{n} \cdot \tfrac{1}{2!} C_{500}^{(2,2)} \cdot (n-1) - \tfrac{1}{2!} C_{500}^{(2,2)} \big\rvert &< \tfrac{1}{1\,500\,000} \,, \\
        \big\lvert \tfrac{1}{n} \cdot \tfrac{1}{2!} C_{500}^{(2,2)} \cdot (n-1) \cdot 2n^{\half} \big(\tfrac{500}{n}\big)^3  - 0 \big\rvert &< \tfrac{1}{15\,000\,000} \,, \\
        \big\lvert \tfrac{1}{n} \cdot \tfrac{1}{2!} \cdot (\log(n) + 1)^2  - 0 \big\rvert &< \tfrac{1}{15\,000} \,.
    \end{align*} 
    Working with the bound \( b(n) \) on \( p_n^\sigma(3)_{\up} \): for \( n \geq 2 \times 10^6 \), we have that
    \begin{align*}
        \big\lvert \tfrac{1}{n^2} \cdot \tfrac{1}{3!} \cdot \tfrac{35(n-1)^2}{16} \zeta(5) - \tfrac{1}{3!} \cdot \tfrac{35}{16} \zeta(5) \big\rvert &< \tfrac{1}{2\,500\,000} \,, \\
        \big\lvert \tfrac{1}{n^2} \cdot \tfrac{1}{3!} \cdot \tfrac{15(n-1)}{4} \zeta(4) \zeta(3) (1 + \log(n)) - 0 \big\rvert &< \tfrac{1}{150\,000} \,, \\
        \big\lvert \tfrac{1}{n^2} \cdot \tfrac{1}{3!} \cdot \tfrac{5}{2} (n-1) \zeta(3) \zeta(2) - 0 \big\rvert &< \tfrac{1}{2\,400\,000} \,, \\
        \big\lvert \tfrac{1}{n^2} \cdot \tfrac{1}{3!} \cdot n (1 + \log(n))^3 - 0 \big\rvert &< \tfrac{1}{3\,000} \,.
    \end{align*}
    
    Interval arithmetic now shows that for \( n \geq 2 \times 10^6 \), we have
    \[
        \frac{a(n)^2}{b(n)} \in  (4.11085\ldots, 4.11916\ldots) \,,
    \]
    and in particular is \( > 4 \).  The upshot now is that for \( n \geq 2\times10^6 \),
    \[
    \frac{p_n^\sigma(2)^2}{p_n^\sigma(1)p_n^\sigma(3)} \geq \frac{a(n)^2}{p_n^\sigma(1) \cdot b(n)} \geq \frac{4}{p_n^\sigma(1)}  \,.
    \]
    So if 
    \[
        \frac{4}{p_n^\sigma(1)} \geq 1 \,, 
    \]
    i.e. \( p_n^\sigma(1) = \sigma_{-1}(n) \leq 4 \), then \( n \) certainly cannot be a counterexample to the log-concavity of \( p_n^\sigma(k) \) at $k=2$.  In other words, any counterexample to the log-concavity of \( p_n^\sigma(k) \) at $k=2$ in the range \( 10^7 \leq n \leq \lambda \), must have \( \sigma_{-1}(n) > 4 \).
    
    \subsection{Elimination of counterexamples with \texorpdfstring{\( 10^7 \leq n < \lambda \)}{10\textasciicircum{}7 <= n < lambda}}

    Counterexamples must have \( \sigma_{-1}(n) > 4 \).  This is OEIS sequence \href{https://oeis.org/A068404}{A068404}, whose first entry is \( 27\,720 \). The density of such numbers is reported to be between \( 0.000176363 \) and \( 0.000679406 \), see the data on \cite{Col10} obtained using the method of Wall et al.~\cite{Wal72} and method of Del\'eglise \cite{Del98}, respectively. This leaves between roughly  \( 0.000176363 \lambda = 11\,501\,426\,231 \) and \( 0.000679406 \lambda = 44\,307\,127\,859 \) candidates to check, saving a factor of \( 10^3 \) over checking the whole range.
    
    We outline a procedure below to enumerate all such \( n \) with \( \sigma_{-1}(n) > 4 \); this procedure can be trivially parallelised and shows there are actually \( 16\,565\,226\,666 \) candidates to check.  The parallel version of our computation in \texttt{gp/pari} took approximately 5697.75 hours = 237.4 days of CPU time, across 48+ threads (on multiple machines), equating to around 5 days of real time. \medskip

    To start, recall that \( \sigma_{-1}(n) \) is multiplicative, so given the prime factorisation \( n = \prod_{i} p_i^{a_i} \), we have
    \[
        \sigma_{-1}(n) = \prod\nolimits_{i} \sigma_{-1}(p_i^{a_i}) \,.
    \]
    For a prime \( p \), we also have
    \begin{equation}\label{eqn:Sigma-1UpperBound}
        \sigma_{-1}(p^a) = \sum_{i=0}^a p^{-i} = \frac{p - p^{-a}}{p-1} < \frac{p}{p-1} \,.
    \end{equation}
    Note that this is a decreasing function for \( p > 1\).

    \begin{rmk}
        Given the inequality in Eqn.~\eqref{eqn:Sigma-1UpperBound}, it is straightforward to show that any number $n$ satisfying $\sigma_{-1}(n) > 4$ has at least 4 distinct prime divisors. When restricting to odd numbers, this amount even increases to at least 21 distinct prime divisors.
    \end{rmk}
    
    Fix a list of the first \( m \) primes \( \{ p_1 = 2, p_2 = 3, \ldots, p_m \} \) and choose exponents \( a_i \) such that \( p_1^{a_1} \cdots p_m^{a_m} \leq \lambda \).  Write \( s = p_1^{a_1} \cdots p_m^{a_m} \).  The question is whether a number of the form \( n =  sn' \), \( (s, n') = 1 \) can have \( \sigma_{-1}(n) > 4 \) and \( n \leq \lambda \).  We can build such a number by selecting the smallest remaining prime (which simultaneously has the largest possible contribution to \( \sigma_{-1} \)) until either \( \sigma_{-1}(s) \cdot \prod_{j=m+1}^{m'} \frac{p_j}{p_j-1} > 4 \), in which case \( s \) is a potentially good starting value, in that some multiple \( n = sn' \), with \( (s,n') = 1 \), below \( \lambda \) can potentially have \( \sigma_{-1}(n) > 4 \).  Or we continue until \( s \cdot \prod_{j=m+1}^{m'} p_j > \lambda \), in which case no multiple \( n = sn' \),  with \( (s,n') = 1 \), of \( s \) below \( \lambda \) can have \( \sigma_{-1}(n) > 4 \).

    \begin{rmk}
        Since \( \prod_{i=1}^m \sigma_{-1}(p_m) = \prod_{i=1}^m (1 + p_m^{-1}) > \sum_{i=1}^m p_m^{-1} \) diverges, one can always find some multiple \( sn' \), with \( (s,n') = 1 \), which has \( \sigma_{-1}(sn') > 4 \).  For our enumeration question, the important point is whether this multiple satisfies \( sn' \leq \lambda \) or not.
    \end{rmk}

    For a fixed list of the first $m$ primes, there are only finitely many possible exponents, i.e. \( 0 \leq a_i \leq \log_{p_i}(\lambda) \), to test for good starting values \( s \leq \lambda \).  Once this list is determined, we enumerate all multiples \( n = sn' \), \( (s, n') = 1 \), with \( n' \leq \lambda/s \), and check whether or not the lower bound on the ratio (with given \( \mu \)) satisfies
    \[
        \frac{(p_n^\sigma(2)_{\low,\mu})^2}{p_n^\sigma(1)p_n^\sigma(3)_{\up}} > 1 \,.
    \]
    If the lower bound is \( > 1 \), then \( n \) cannot be a counterexample to the log-concavity of \( P_n^\sigma(X) \) at $k=2$.  Checking these bounds for sufficiently large \( \mu \) should determine whether or not it really is the case that
    \[
     \frac{p_n^\sigma(2)^2}{p_n^\sigma(1)p_n^\sigma(3)} < 1 \,,
    \]
    and so whether or not \( n \) is a counterexample to log-concavity of \( P_n^\sigma(X) \) at $k=2$.  
    
    \begin{rmk}
        The step of iterating through all multiples \( n = sn' \), \( (s,n') = 1 \) can be done in parallel as each \( s \) is independent.  Alternatively, for small values of \( s \), it is actually better to search for \( sn' \) in parallel over the individual residue classes \( n' \modi{\lcm(p_1, \ldots, p_m)} \) with \( (n',p_1 p_2  \cdots p_m) = 1 \), as this makes the parallelisation finer grained.
        \end{rmk}

    We implemented the procedure with \( m = 6 \), and initial list of primes \( \{ 2, 3, 5, 7, 11, 13 \} \), using \texttt{gp/pari}.  We found a total of \( 123366 \) potential starting values \( s \).  After enumerating all multiples \( sn' \leq \lambda \), and testing with \( \mu = 500 \), we found only 4 potential counterexamples \( n \leq \lambda \) 
    \[
        32\,607\,253\,879\,200, \; 55\,898\,149\,507\,200, \;59\,753\,194\,300\,800, \; 65\,214\,507\,758\,400 \,.
    \]
    Testing again with \( \mu = 5\,000 \) eliminates the first 3, showing that \( \lambda \) is indeed the first counterexample.  This completes the proof of Theorem \ref{thm:result}, part ii). \hfill \qedsymbol

\section{Positive density of counterexamples}

For our last result, we will make use of another simple bound.
\begin{lemma}\label{lem:BoundOnRatio}
    For \( n \geq 2 \times 10^6 \), it holds that
    \[\frac{p_n^\sigma(2)^2}{p_n^\sigma(1) p_n^\sigma(3)} < \frac{6.20927\ldots}{\sigma_{-1}(n)} \, .\]
\end{lemma}

\begin{proof}
     Fix \( \mu = 1000 \), with \( C_{1000}^{(2,2)} = 2.49726\ldots \), and \( C_{1000}^{(4,2)} = 1.74895\ldots \).  Since \( \sigma_{-3}(n) < \zeta(3) \), we note that
     \[
         p_n^\sigma(2)_{\up}  < c(n) \coloneqq \frac{1}{2!} \left(\frac{5}{2} (n-1) \zeta(3) + (1 + \log(n))^2\right) \,.
     \]
     Likewise, since \( \sigma_{-5}(n) > 1 \), and \( \sigma_{-4}(n) < \zeta(4) \)
     \[
            p_n^\sigma(3)_{\low,1000}  > d(n) \coloneqq \frac{1}{3!}\bigg(\begin{aligned}[t] \frac{(n-1)^2}{2} & C_{1000}^{(2,2)} C_{1000}^{(4,2)} - (n-1)^2 n^{\half} \Big(\frac{1000}{n}\Big)^5 C_{1000}^{(2,2)}  C_{1000}^{(4,2)}\\
    & {} - \frac{3(n-1)}{2} C_{1000}^{(2,2)}  \zeta(4) \zeta(3) (1 + \log(n))\\
    &{} - C_{1000}^{(2,2)} \cdot (1000)^2 \zeta(2) (1 + \log(n)) - n(1+\log(n))^3\bigg) \,. \end{aligned}
     \]
     As in Section \ref{sec:counterprop}, we compute that
     \[
        \lim_{n\to\infty} \frac{c(n)^2}{d(n)} = \frac{75 \zeta(3)^2}{4 C_{1000}^{(2,2)} C_{1000}^{(4,2)}} = 6.20309329901722084305\ldots \,.
     \]
     In particular the ratio is eventually in an arbitrarily small interval around this limit.  We can check for \( n \geq 2 \times 10^6 \), that the the following holds
     \[
        \frac{c(n)^2}{d(n)} \in (6.19692\ldots, 6.20927\ldots) \,.
     \]
     So for \( n \geq 2 \times 10^6 \), we have
     \[
        \frac{p_n^\sigma(2)^2}{p_n^\sigma(1) p_n^\sigma(3)} < 
        \frac{c(n)^2}{p_n^\sigma(1) \cdot d(n)} < \frac{6.20927\ldots}{\sigma_{-1}(n)} \,,
     \]
     as claimed.
\end{proof}

We are now in a position to finalise our proofs.

\begin{proof}[Proof of Theorem \ref{thm:density}]
    The reasoning in Section \ref{sec:minimality} established that any possible counterexample $n$ to the log-concavity has to satisfy $\sigma_{-1}(n) > 4$. As mentioned there, the asymptotical density of these numbers is at most 0.000679406.

    Conversely, due to Lemma \ref{lem:BoundOnRatio}, for any $n$ satisfying $\sigma_{-1}(n) > 6.20927\ldots $ the polynomial $P_n^\sigma(X)$ is not log-concave at $k = 2$.  The smallest natural numbers $n$ to satisfy this is the 93\textsuperscript{rd} superabundant number $S_{93} = 4\,043\,299\,481\,020\,800$ and hence the claim follows because $\sigma_{-1}(kn) > \sigma_{-1}(n)$ for any $k \in \N$.  In particular, \( \delta_{\mathrm{inf}} \geq \frac{1}{S_{93}} = 2.47323 \ldots \times 10^{-16} \).
\end{proof}

\begin{rmk}
    As \( \mu \to \infty \), we can reduce the constant in upper bound in Lemma \ref{lem:BoundOnRatio} to a small interval around
    \[
        \frac{75\zeta(3)^2}{4 \lim_{\mu \to \infty} C_\mu^{(2,2)} C_\mu^{(4,2)}} = \frac{30 \zeta(3)^2}{7} = 6.19260342185843243105\ldots \,.
    \]
    As the 92\textsuperscript{nd} superabundant number \( S_{92} = 3\,032\,474\,610\,765\,600 \) satisfies
    \[
        \sigma_{-1}(S_{92}) = 6.18956\ldots < \frac{30 \zeta(3)^2}{7} \,,
    \]
    we cannot readily improve the asymptotic density bound further, without more involved work.
\end{rmk}

\bibliographystyle{habbrv2}
\bibliography{bibliography.bib}

\end{document}